\documentclass{amsart}[10pt]
\usepackage{amsmath,a4wide}
\usepackage{amssymb} 
\usepackage{amsthm}
\usepackage{graphicx}
\usepackage[all]{xypic}
 \newtheorem{theorem}{Theorem}[section]
 \newtheorem{corollary}{Corollary}[section]
 \newtheorem{proposition}{Proposition}[section]
 \newtheorem{lemma}{Lemma}[section]
 \newtheorem{claim}{Claim}[section]
 {\theoremstyle{definition}
 \newtheorem{remark}{Remark}[theorem]}
  {\theoremstyle{definition}
 \newtheorem{definition}{Definition}[section]}
  {\theoremstyle{definition}
 }
 {\theoremstyle{definition}
 }

 \makeatletter
    
    \@addtoreset{equation}{section}
  \makeatother
\begin{document}  

 \newcommand{\Z}{ \mathbb{Z}}
 \newcommand{\Q}{ \mathbb{Q}}
 \newcommand{\R}{ \mathbb{R}}

 \newcommand{\GL}{ \textrm{GL}}
 \newcommand{\mR}{ \mathcal{R}}
 \newcommand{\Aut}{ \textrm{Aut}\,}
 \newcommand{\Inn}{ \textrm{Inn}\,}
 \newcommand{\red}{\textsf{red}\,}
 \newcommand{\Hom}{\textrm{Hom}}
 \newcommand{\otau}{\widetilde{\tau}}
 \newcommand{\mA}{ \mathbf{A}}
 \newcommand{\mB}{ \mathbf{B}}
 \newcommand{\mC}{ \mathbf{C}}
 \newcommand{\mD}{ \mathbf{D}}
 \newcommand{\wTheta}{\widehat{\Theta}\,}

\title[$n$-braid action on the free group of rank $n$]{Actions of the n-strand braid groups on the free group of rank n which are similar to the Artin representation}
\thanks{T.I. was partially supported by the Grant-in-Aid for Research Activity start-up, Grant Number 25887030.}
\author{Tetsuya Ito}
\address{Research Institute for Mathematical Sciences, Kyoto University
Kyoto University, Sakyo-ku, Kyoto, Japan}
\email{tetitoh@kurims.kyoto-u.ac.jp}
\urladdr{http://www.kurims.kyoto-u.ac.jp/~tetitoh/}

\begin{abstract}
We classify an action of the $n$-strand braid group on the free group of rank $n$ which is similar to the Artin representation in the sense that the $i$-th generator $\sigma_{i}$ of $B_{n}$ acts so that it fixes all free generators $x_{j}$ except $j = i,i+1$. We determine all such representations and discuss knot invariants coming from such representations.
\end{abstract}
\maketitle

\section{Introduction}

Let $B_{n}$ be the braid group of $n$-strands, with standard generators $\sigma_{1},\ldots,\sigma_{n-1}$ and let $F_{n}$ be the free group of rank $n$ generated by $x_{1},\ldots,x_{n}$. Throughout the paper, we consider {\em right actions}.

The {\em Artin representation} $\rho_{\sf Artin}$ is a standard action of $B_{n}$ on $F_{n}$, given by
 \begin{gather*}
 (x_{j}) \rho_{\sf Artin}(\sigma_{i}) = 
\begin{cases}
x_{i}x_{i+1}x_{i}^{-1} & (j=i),\\
x_{i} & (j=i+1), \\
x_{j} & \text{(Otherwise)}.
\end{cases}
\end{gather*} 
A remarkable feature of $\rho_{\sf Artin}$ is that $\rho_{\sf Artin}(\sigma_i)$ preserves all free generators of $F_{n}$ except $x_{i}$ and $x_{i+1}$, and that $\rho_{\sf Artin}(\sigma_i)$ sends $x_{i}$ and $x_{i+1}$ to words on $\{x_{i},x_{i+1}\}$. 
In this paper we study an $\Aut(F_{n})$-representation of the braid group $B_{n}$ having the same property.

For $i=1,\ldots,n-1$, we say an automorphism $\otau: F_{n} \rightarrow F_{n}$ is {\em $i$-local} if $\otau(x_{j}) = x_{j}$ for $j \neq i,i+1$ and $\otau ( \langle x_{i},x_{i+1} \rangle ) = \langle x_{i},x_{i+1} \rangle$ where $ \langle x_{i},x_{i+1} \rangle$ represents the subgroup of $F_{n}$ generated by $x_{i}$ and $x_{i+1}$. In other words, $\otau \in \Aut(F_{n})$ is $i$-local if and only if there exists $\tau \in \Aut(F_{2})$ such that
\[ \otau = \textsf{id} * \tau * \textsf{id} : F_{n}= F_{i-1}*F_{2}*F_{n-i-1} \rightarrow F_{i-1}*F_{2}*F_{n-i-1} = F_{n}. \]
We say $\tau$ is the {\em core} of $\otau$, and write $\otau = T^{i}(\tau)$.

\begin{definition}
An $\Aut(F_{n})$ representation of the braid group $\Theta: B_{n} \rightarrow \Aut(F_{n})$ is {\em local} if $\Theta(\sigma_{i})$ is $i$-local for all $i=1,2,\ldots,n-1$. 
\end{definition}

For a local $\Aut(F_{n})$ representation $\Theta$, we denote the core of $\Theta(\sigma_{i})$ by $\tau^{\Theta}_{i}$ (or simply $\tau_{i}$), and we will say that $(\tau_{1},\ldots,\tau_{n-1}) \in \Aut(F_{n})^{n-1}$ {\em defines} $\Theta$.

A local $\Aut(F_{n})$ representation with $\tau_{1}= \tau_{2} = \cdots = \tau_{n-1}$ was introduced by Wada \cite{w} and called a {\em Wada-type representation}. 
In \cite[Corollary 1.2]{i} we determined all Wada-type representations (actually we have classified more general objects, called a solution of certain variant of set-theoretical Yang-Baxter equation).
We proved that, up to certain natural symmetries, there are exactly seven types of Wada-type representations as Wada conjectured in \cite{w}.

In this paper we give a classification of local $\Aut(F_{n})$-representations: 
In Theorem \ref{thm:main} we list up all local $\Aut(F_{3})$-representations, and give an oriented, labelled graph whose edge-path describes all local $\Aut(F_{n})$-representations.

One motivation of studying and introducing a local $\Aut(F_{n})$ representation comes from knot theory. By imitating the presentation of knot groups using the Artin representations \cite{b}, one obtains a group-valued invariant of knots and links from Wada-type representations \cite{w}. By generalizing the Artin representation in a different manner, Crisp and Paris constructed a group-valued invariant of knots and links \cite{cp}, and in \cite{i1} the author used a quandle variant of Crisp-Paris' construction to define a quandle-valued knot invariant. However, there are few applications of these invariants. Unfortunately, the classification of Wada-type representations shows that in most cases the group-valued invariants from Wada-type representations are determined by the fundamental group of the double branched covering of knots, so they are less interesting. 

We can use a local $\Aut(F_{n})$-representation to construct a group-valued invariant in a similar manner (Proposition \ref{prop:ginvariant}). Unfortunately, our classification again shows that group-valued invariants from local $\Aut(F_{n})$-representations are nothing new, because they coincide with the group-valued invariants from a Wada-type representation of the same type (see Theorem \ref{theorem:ginvariant}). 

Nevertheless, the classification and idea of local $\Aut(F_{n})$ representations are interesting in its own right. First, a local $\Aut(F_{n})$ representation can be seen as a generalization of Yang-Baxter equation.
 Second, an idea of local $\Aut(F_{n})$ representations to produce knot invariants can be investigated further, although the local $\Aut(F_{n})$ representation themselves cannot produce new interesting invariants. In the theory of quandles or state-sum invariants (see \cite{cjkls}), one treats all crossing in the same manner, in the sense that we assign the relations or weights of the same form at each crossings. It is an interesting problem to try to generalize such invariants by giving \emph{different} types of relations or weights at each crossing. Finally, investigating a homomorphism from $B_{n}$ to $\Aut(F_{n})$  would be interesting and will help us to understand relationships between these two groups.

\section{Classification of local $\Aut(F_{n})$ braid representations}
\label{sec:classification}

\subsection{Statement of the Classification theorem}

For $\tau,\kappa \in \Aut(F_{2})$ and $i,j \in \{1,\ldots,n-1\}$, $T^{i}(\tau)T^{j}(\kappa) = T^{j}(\kappa)T^{i}(\kappa)$ if $|i-j|>1$, so $(\tau_{1},\ldots, \tau_{n-1})\in \Aut(F_{2})^{n-1}$ defines a local $\Aut(F_{n})$ representation if and only if $(\tau_{i},\tau_{i+1})$ defines a local $\Aut(F_{3})$ representation for each $i$.
Hence to classify local $\Aut(F_{n})$ representations, it is sufficient to consider the case $n=3$. From now on, we study the problem when $(\tau,\kappa) \in \Aut(F_{2})^{2}$ defines a local $\Aut(F_{3})$ braid representation.

Let $F_{2}$ be the free group of rank two generated by $\{a,b\}$.
For a pair of automorphisms $(\tau,\kappa) \in \Aut(F_{2})^{2}$ that defines a local $\Aut(F_{3})$ representation $\Theta$, let $A,B,C,D$ be the reduced words representing $\tau(a),\tau(b),\kappa(a),\kappa(b)$, respectively. We will often say reduced words $(A,B,C,D)$ {\em define} a local $\Aut(F_{3})$ braid representation.

To give a precise statement of the classification, first we observe symmetries of local $\Aut(F_{3})$ representation derived from symmetries of the free groups and the braid groups, which we will call \emph{natural symmteries}.

In this paper we use the following notation. For words $V$ and $W$ we will write $V \equiv W$ if they are the same as words, and will write $V=W$ if they represents the same element in the free group. Thus, $abb^{-1} \not \equiv a$ but $abb^{-1} = a$, for example.\\

\begin{description}
\item[Inverse symmetry]
 {$ $}\\
Let $\tau'$ and $\kappa'$ be the core of $\Theta(\sigma_{1}^{-1})$ and $\Theta(\sigma_{2}^{-1})$ respectively.
Then $(\tau',\kappa')$ also defines a local $\Aut(F_{3})$ representation $\Theta^{-}$ which we call the {\em inverse representation} of $\Theta$.\\

\item[Swap symmetry]
 {$ $}\\
Let $A^{\sigma}$ be the word over $\{a^{\pm 1},b^{\pm 1}\}$ obtained by interchanging the letters $a$ and $b$. For example, $A^{\sigma} \equiv b^{-2}ab$ if $A \equiv a^{-2}ba$. $B^{\sigma}, C^{\sigma}$ and $D^{\sigma}$ are defined similarly.
Then $(D^{\sigma},C^{\sigma},B^{\sigma},A^{\sigma})$ also defines a local $\Aut(F_{3})$ representation $\Theta^{\sigma}$ which we call the {\em swap-dual representation} of $\Theta$.\\

\item[Backward symmetry]
 {$ $}\\
Let $\overleftarrow{A}$ be a reduced word obtained by reading $A$
backward. For example, $\overleftarrow{A} \equiv a^{2}bab^{-1}$ if $A \equiv b^{-1}aba^{2}$. $\overleftarrow{B}$, $\overleftarrow{C}$ and $\overleftarrow{D}$ are defined similarly. Then $(\overleftarrow{A}, \overleftarrow{B}, \overleftarrow{C}, \overleftarrow{D})$ defines a local $\Aut(F_{3})$ braid representation $\overleftarrow{\Theta}$ which we call the {\em backward-dual representation} of $\Theta$. \\

\end{description}

It is directly checked that these three symmetries commute each other. That is, $(\Theta^{-})^{\sigma}$, the swap-dual of the inverse $\Theta^{-}$ coincide with the $(\Theta^{\sigma})^{-}$, the inverse of the swap-dual $\Theta^{\sigma}$. 
So we simply write $\Theta^{-\sigma}$ to represent the swap and inverse dual of $\Theta$, for example. Thus for each local $\Aut(F_{3})$ representation $\Theta$, there may be eight different local $\Aut(F_{3})$-representations by considering natural symmetries.

The following is the main result of this paper.

\begin{theorem}[Classification of local $\Aut(F_{3})$ braid representation.]
\label{thm:main}
Up to natural symmetries, reduced words $(A,B,C,D)$ on $\{a^{\pm1},b^{\pm1}\}$ defines a local $\Aut(F_{3})$ braid representation if and only if $(A,B,C,D)$ is one of the following.
\begin{enumerate}
\item[($T$)] $(A,B,C,D)=(a,b,a,b)$.
\item[($T'$)] $(A,B,C,D)=(a,b^{-1},a^{-1},b)$.
\item[($A_{r}\,$-- 1)]  $(A,B,C,D)=(a^{r}b a^{-r}, a, a^{r}ba^{-r}, a)$ \ $(r \in \Z_{\geq 0})$.
\item[($A_{r}\,$-- 2)]  $(A,B,C,D) = (a^{r}ba^{-r}, a, a^{r} b^{-1}a^{-r},a^{-1})$ \ $(r \in \Z_{\geq 0})$.
\item[($A_{r}\,$-- 3)]  $(A,B,C,D) =  (a^{r}b^{-1}a^{-r}, a^{-1}, a^{-r} b^{-1}a^{r}, a^{-1})$ \ $(r \in \Z_{\geq 0})$.

\item[($B\,$-- 1)] $(A,B,C,D) = (b^{-1},a , b^{-1}, a)$.
\item[($B\,$-- 2)] $(A,B,C,D) = (b^{-1},a , b, a^{-1})$. 

\item[($C\,$-- 1)] $(A,B,C,D) = (ab^{-1}a,a,ab^{-1}a,a)$.
\item[($C\,$-- 2)] $(A,B,C,D) = (ab^{-1}a, a, aba, a^{-1})$.
\item[($C\,$-- 3)] $(A,B,C,D) = (aba, a^{-1}, aba, a^{-1})$.

\item[($D\,$-- 1)] $(A,B,C,D) = (a^{-1}b^{-1}a,b^{2}a,a^{-1}b^{-1}a,b^{2}a)$.
\item[($D\,$-- 2)] $(A,B,C,D) = (aba^{-1},b^{2}a^{-1},a^{-1}b^{-1}a,b^{2}a)$.
\item[($D\,$-- 3)] $(A,B,C,D) = (a^{-1}b^{-1}a,b^{2}a,a^{-1}ba,a^{-1}b^{2})$.
\item[($D\,$-- 4)] $(A,B,C,D) = (aba^{-1},b^{2}a^{-1},a^{-1}ba,a^{-1}b^{2})$.

\end{enumerate}

\end{theorem}

It is convenient to express the classification of local $\Aut(F_{n})$ representations by using an oriented graph $\Gamma$ in Figure \ref{fig:classification}. 
For each vertex $v$ of $\Gamma$, we assign an element of $\Aut(F_{2})$, $\tau_{v}: F_{2} \rightarrow F_{2}$ by indicating a pair of reduced words $(\tau_v(a),\tau_v(b))$. Two vertices $v$ and $w$ are connected by an oriented edge $e$ oriented from $v$ to $w$ if and only if $(\tau_{v},\tau_{w})$ defines a local $\Aut(F_{3})$ representation. Up to natural symmetries, $(\tau_v,\tau_w)$ appears in the list in Theorem \ref{thm:main} so for each edge, for convenience we assign a labeling to indicate the correspondence to the list in Theorem \ref{thm:main}: for example, $\overleftarrow{(A_{r}-1)}$ means the backward-dual of the local $\Aut(F_{3})$-representation $(A_{r}-1)$ in Theorem \ref{thm:main}.

Up to natural symmetries, the set of local $\Aut(F_{n})$ representations is identified with the set of oriented edge-path of length $(n-2)$ in $\Gamma$: For an oriented edge-path $\gamma$ of $\Gamma$, assume that $\gamma$ passes the vertices $v_{1},v_{2},\ldots,v_{n-1}$ in this order. Then $(\tau_{v_{1}},\ldots,\tau_{v_{n-1}}) \in \Aut(F_{2})^{n-1}$ defines a local $\Aut(F_{n})$ representation $\Theta_{\gamma}$. Conversely, for a local $\Aut(F_{n})$ representation $\Theta$, take a vertex $v_{i}$ $(i=1,\ldots,n-1)$ of $\Gamma$ so that $\tau^{\Theta}_{i} = \tau_{v_{i}}$ $i=1,\ldots,n-1$. If we take $\gamma$ an oriented edge-path $\gamma$ of $\Gamma$ so that $\gamma$ passes the vertices $v_{1},\ldots,v_{n-1}$ in this order, then $\Theta = \Theta_{\gamma}$.

For each connected component of $\Gamma$, we assign the labeling $\{T,T',A_{r} (r\in \Z_{\geq 0}),C,D\}$ and we call a local $\Aut(F_{n})$ representation $\Theta$ is of type $\mathcal{L}$ if $\Theta$ corresponds to the oriented edge-path in the connected component labelled by $\mathcal{L} \in \{T,T',A_{r},B,C,D,\}$.  

We also note:
\begin{theorem}
A local $\Aut(F_{n})$-representation which is not of type $(T)$, $(T')$, $(A_{0})$, $(B)$ is faithful.
\end{theorem}
\begin{proof}
The faithfulness of Wada-type representation is proven in \cite{sh}, by using the  theory of Dehornoy's ordering \cite{ddrw}: For every non-trivial braid $\beta$, either $\beta$ or $\beta^{-1}$ is represented by a word $w$ over $\{\sigma_{1}^{\pm 1},\ldots,\sigma_{n-1}^{\pm 1}\}$ which is $\sigma$-positive (there exists $i\in \{1,\ldots,n-1\}$ such that $w$ contains at least one $\sigma_{i}$ but does not contain any $\sigma_{1}^{\pm 1},\ldots,\sigma_{i-1}^{-1},\sigma_{i}^{-1}$.

Let $\Theta$ be a local $\Aut(F_{n})$ representation which is not of type $(T)$, $(T')$, $(A_{0})$, $(B)$ (note that a local $\Aut(F_{n})$ representation of type $(T)$, $(T')$, $(A_{0})$, $(B)$ are clearly non-faithful). A direct calculation (see \cite{sh}, for details) shows that for a braid $\beta$ represented by a $\sigma$-positive word, $\Theta(\beta) \neq 1$.
\end{proof}

\begin{figure}
\[
[T] \ \ \ 
\xymatrix{
(a,b) \ar@(ul,ur)^{(T)}
}
\ \ \ \ \ \ \ \ [T'] \ \  
\xymatrix{
(a,b^{-1}) \ar[r]^{(T')} & (a^{-1},b)
}
\]

\vspace{0.3cm}

\[
\lower4ex\hbox{  $[A_r]$}
\ \ \ \ 
\xymatrix{
 (a^{r}ba^{-r},a) \ar@(ul,ur)^{(A_{r}-1)} \ar[rr]^{(A_{r}-2)}& & (a^{r}b^{-1}a^{-r},a^{-1}) \ar[d]^{(A_{r}-2)^{-\sigma}} \ar@/^2mm/[lld]^<<<<<<<{\;\;\;(A_{r}-3)}\\
 (a^{-r}b^{-1}a^{r},b) \ar[u]^{\overleftarrow{(A_{r}-2)}{}^{-\sigma}} \ar@/^3mm/[urr]^<<<<<{\;\;\;\overleftarrow{(A_{r}-3)}} & &
 (a^{-r}ba^{r},b) \ar[ll]^{\overleftarrow{(A_{r}-2)}}
\ar@(dl,dr)_{\overleftarrow{(A_{r}-1)}} \\
}\]

\[
[B] \ \ \
\xymatrix{
(b^{-1},a) \ar@(ul,ur)^{(B-1)}  \ar@/^2mm/[rr]^{(B-2)} &  & (b,a^{-1}) \ar@(ur,ul)_{(B-1)^{\sigma}} \ar@/^2mm/[ll]^{(B-2)^{-}}
}\]

\[\xymatrix{
[C] \ \ \ 
(ab^{-1}a,a) \ar@(ul,ur)^{(C-1)}  \ar@/^2mm/[rr]^{(C-2)} &  & (aba,a^{-1}) \ar@(ur,ul)_{(C-3)} \ar@/^2mm/[ll]^<<<<<{(C-2)^{-\sigma}}
}\]

\vspace{0.3cm}

\[
\lower9ex\hbox{$[D]$} \ \ \ 
\xymatrix{
  & (a^{-1}ba,b^{2}a) \ar@(ul,ur)^{(D-1)} \ar[rd]^{\;\;\overleftarrow{(D-2)}{}^{-\sigma}}&\\
(aba^{-1},b^{2}a^{-1}) \ar[ur]^{\overleftarrow{(D-3)}{}^{-\sigma}} 
\ar@/^3mm/[rr]^{\overleftarrow{(D-4)}{}^{-\sigma}}
& & (a^{-1}ba,a^{-1}b^{2})\ar@/^3mm/[ll]^{(D-4)^{-\sigma}} \ar[dl]^{(D-3)^{-\sigma}} \\
 & (ab^{-1}a^{-1},ab^{2}) \ar@(dl,dr)_{(D-1)^{-\sigma}} \ar[ul]^{(D-2)^{-\sigma} \; \;}&
}\]
\caption{Graph $\Gamma$ describing the clasisfication of local $\Aut(F_{n})$-representations}
\label{fig:classification}
\end{figure}
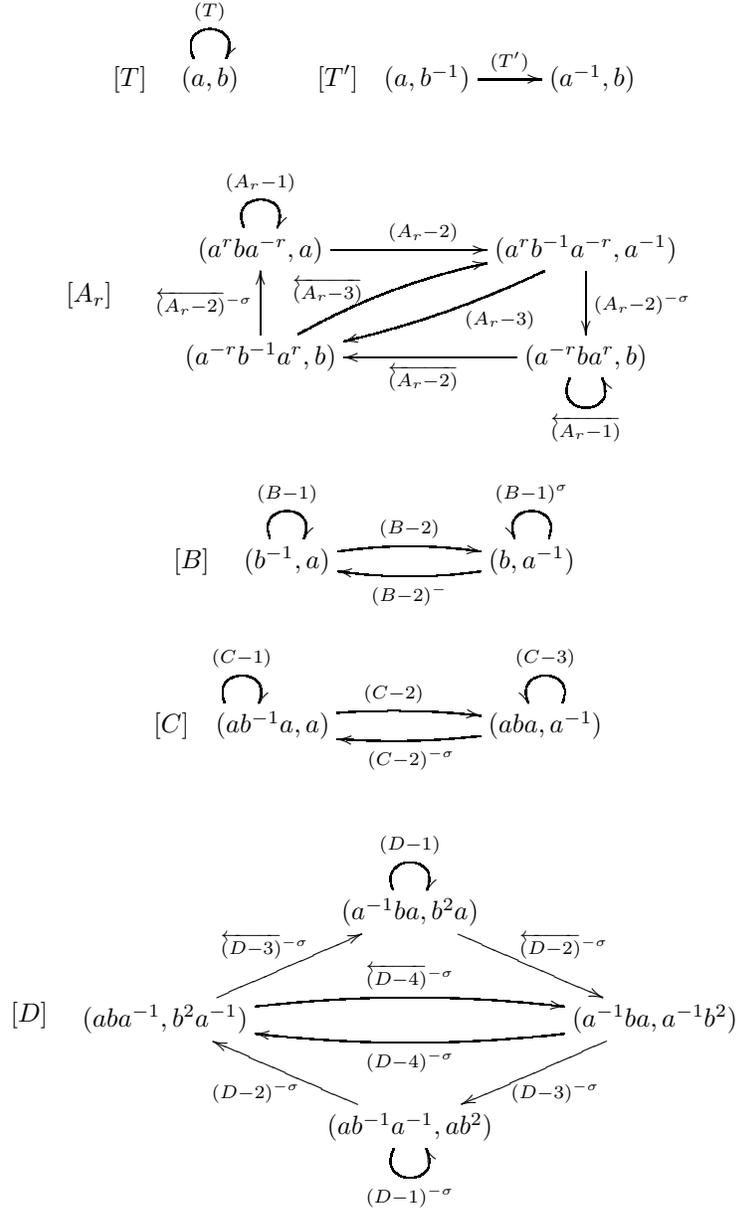

\subsection{Classification: Outline}

Before proceeding to prove the classification theorem, we briefly explain an outline of the proof.

First of all, we review our notations.
A word $W$ over an alphabet $\mathcal{A}$ is \emph{reduced} if $W$ contains no subword of the form $aa^{-1}$ and $a^{-1}a$ $(a \in \mathcal{A})$. By removing subwords of the form $aa^{-1}$ or $a^{-1}a$, every word $W$ is changed to the unique reduced word which we denote by $\red(W)$. For two words $W$ and $V$, $W=V$ if and only if $\red(W) \equiv \red(V)$.

For a word $W$ over an alphabet $\{a^{\pm1},b^{\pm1}\}$, and words $A$ and $B$ over another alphabet $\mathcal{X}=\{x^{\pm 1}, y^{\pm 1}\ldots \}$, $W(A,B)$ represents the word over $\mathcal{X}$ obtained from $W$ by substituting each letter $a^{\pm 1}$ and $b^{\pm 1}$ with the word $A^{\pm 1}$ and $B^{\pm 1}$, respectively. For example, if $W \equiv aab$, $A \equiv xy$ and $B \equiv yx^{-1}$, then $W(A,B)\equiv xyxyyx^{-1}$.

For a word $W$ over $\{x^{\pm 1},y^{\pm 1},z^{\pm 1}, \ldots \}$, we will write 
\[ W \equiv x^{k_1}y^{k_2}*\cdots \]
to say that $W$ is written as $W \equiv x^{k_1}y^{k_2}W'$, where $W' \not \equiv yW'', y^{-1}W''$. (Here $W'$ might be empty word.)
On the other hand, we will simply use dots, like
\[ W \equiv x^{k_1}y^{k_2} \cdots \]
to say that $W$ is written as $W \equiv x^{k_1}y^{k_2}W'$.
For example, $W \equiv x \cdots x$ just says that $W \equiv xW'x$ for some word $W'$, but if we write $W \equiv x * \cdots * x$, then it means that $W \equiv xW'x$ with $W' \not \equiv x^{\pm 1}W'',W''x^{\pm 1}$.

The proof of the classification theorem goes as follows. By direct computation, $(A,B,C,D)$ defines a local $\Aut(F_{3})$ representation if and only if the following four conditions are satisfied.

\begin{enumerate}
\item[{\bf[T]}:] $A ( A(x,y), C( B(x,y),z)) =  A (x, C(y,z) )$.
\item[{\bf[M]}:] $B ( A(x,y), C( B(x,y), z) ) = C ( B(x, C(y,z) ), D(y,z) )$.
\item[{\bf[B]}:] $D (B(x,y), z) = D ( B(x, C(y,z) ), D(y,z) )$.
\item[{\bf[Aut]}:] The homomorphism $\tau$ of $F_{2}=\langle a,b\rangle$ defined by $a \mapsto A$, $b \mapsto B$ is an automorphism. Similarly, the homomorphism $\kappa$ of $F_{2}$ defined by $a \mapsto C$, $b \mapsto D$ is an automorphism.
\end{enumerate}

In particular, the last condition {\bf[Aut]} implies that all elements $A,B,C,D$ are primitive, that is, they admit no non-trivial root in $F_{2}$.

Let $(A,B,C,D)$ be reduced words which defines a local $\Aut(F_{3})$ representation $\Theta$. The proof of classification theorem splits into the following two propositions which we will be proven in Section \ref{sec:proof}.

\begin{proposition}
\label{prop:caseI} If $B$ does not contain $a^{\pm 1}$, then $\Theta$ is of type  $(T)$ or $(T')$. 
\end{proposition}

\begin{proposition}
\label{prop:caseIII}
If $B$ contains at least one $a^{\pm 1}$, then up to natural symmetries, $\Theta$ is of type $(A_{r})$, $(B)$, $(C)$ or $(D)$.
\end{proposition}

Finally we explain reasons why we are able to solve the equations {\bf[T]}, {\bf[M]}, and {\bf[B]}. An important feature of equations {\bf[T]} and {\bf[B]} is that the order of the length of words are different.

The left side of {\bf[T]} is a word of the form $W(W(W(\bullet,\bullet),\bullet),\bullet)$, where $W$ is one of $A,B,C,D$. This implies that the length of the left side of  {\bf[T]} is order of $l^{3}$, where $l$ denotes the length of $W$.
On the other hand, the right side of {\bf[T]} is a word of the form $W(W(\bullet,\bullet),\bullet)$ so its length is order of $l^{2}$. This observation suggests that either $l$ is small or the left side of the word {\bf[T]} must admit many cancellations -- as the length $l$ increases, the number of necessary cancellations must grow quite rapidly. 
On the other hand, the number of possible cancellations are limited because all words $A,B,C,D$ are reduced. Thus {\bf[T]} and {\bf[B]} implicitly and vaguely says that the length of the words cannot be large, or these words have strong symmetries which leads to many cancellations.  
  
\begin{remark}
In \cite{i}, the author classified, the solutions of certain variant of the set-theoretical Yang-Baxter equation, which is essentially equal to Wada-type representations without requiring invertibility, the condition {\bf[Aut]}.
The main idea and the method in the proof Theorem \ref{thm:main} are almost the same as the proof of \cite[Theorem 1.1]{i}. However, thanks to {\bf[Aut]} we are often able to simplify arguments. We also remark that the classification of local $\Aut(F_{n})$ representations does not use the classification of Wada-type representations.
\end{remark}

\section{Proof of Classification theorem}
\label{sec:proof}
In this section we prove Theorem \ref{thm:main}, by proving Proposition \ref{prop:caseI} and Proposition \ref{prop:caseIII}.
Let $(A,B,C,D)$ be reduced words that defines a local $\Aut(F_{3})$-representation.

\subsection{Key observation}

Let us regard two words $C(y,z)$ and $D(y,z)$ as alphabets $\mC$ and $\mD$.
Temporary we regard the word $C(B(x,C(y,z)),D(y,z))$, the right side of {\bf[M]}, as a word over $\{x^{\pm 1}, \mC^{\pm 1}, \mD^{\pm 1} \}$.

Then we write $C(B(x,C(y,z)),D(y,z)) = C(B(x,\mC),\mD)$ in the form
\[ C(B(x,\mC),\mD) =  W_{0}(\mC,\mD) x^{n_{1}} W_{1}(\mC,\mD) x^{n_{2}}*\cdots *x^{n_{k}}W_{k}(\mC,\mD),\]
where $W_{i}$ and $n_{i}$ satisfy the following conditions.
\begin{enumerate}
\item[(i)] $W_{0}(\mC ,\mD)$ and $W_{k}(\mC,\mD)$ are reduced words over $\{\mC^{\pm 1},\mD^{\pm 1}\}$.
$W_{0}(\mC,\mD)$ is an empty word if and only if $B \equiv a^{\pm 1} \cdots$ and $C\equiv a^{\pm 1}\cdots$.
\item[(ii)] For $i=1,\ldots,k-1$, $W_{i}(\mC,\mD)$ is a reduced, non-empty word over $\{\mC^{\pm 1},\mD^{\pm 1}\}$.
\item[(iii)] $n_{1},\ldots,n_{k}$ are non-zero integers.
\end{enumerate}

Now we back to see $C(B(x,\mC),\mD)$ as a word over $\{x^{\pm 1}, y^{\pm 1}, z^{\pm 1}\}$, by substituting $\mC$ and $\mD$ with $C(x,y)$ and $D(x,y)$.
Let $W_{i}\equiv W_{i}(C(y,z),D(y,z))$ be a word over $\{ y^{\pm 1}, z^{\pm 1}\}$.
By \textbf{[Aut]}, $C(y,z)$ and $D(y,z)$ generate $F_{2}=\langle y,z\rangle$. Since $W_{i}(\mC,\mD)$ is a non-empty reduced word, as an element of $F_{2}=\langle y,z\rangle$, $W_{i} \neq 1$. This shows the following.

\begin{lemma}
\label{lemma:key}
\[ \red( C(B(x,C(y,z)),D(y,z)) ) \equiv \red(W_{0}) x^{n_{1}} \red(W_{1}) x^{n_{2}}*\cdots * x^{n_{k}} \red(W_{k}) \]
\end{lemma}

A similar calculation applies for $D(B(x,C(y,z)),D(y,z)))$.
Similarly, by regarding $A(x,y)$ and $B(x,y)$ as alphabets $\mA$ and $\mB$, we can apply a similar calculation for $A(A(x,y),C(B(x,y),z))$ and $B(A(A(x,y),C(B(x,y),z)))$.
We will often use (sometimes without referring) Lemma \ref{lemma:key} to compute reduced word expression of left or right sides of equations {\bf [T]}, {\bf [M]} and {\bf [B]}.

\subsection{The case $B$ contains no $a^{\pm 1}$}

\begin{proof}[Proof of Proposition \ref{prop:caseI}]
By {\bf[Aut]} and the hypothesis that $B$ does not contain $a^{\pm 1}$, $B \equiv b^{\beta}$ $(\beta=\pm 1)$.
By {\bf[Aut]}, this implies that $A$ must be of the form
\[ A \equiv b^{p}a^{\alpha} b^{p'} \;\;\; (p,p' \in \Z, \alpha \in \{\pm 1\}) . \]

By {\bf[T]}, we get 
\[ C(y^{\beta},z)^{p} (y^{p}x^{\alpha}y^{p'})^{\alpha}C(y^{\beta},z)^{p'} = C(y,z)^{p} x^{\alpha} C(y,z)^{p'}. \]

By comparing the exponent of $x$, we get $\alpha^{2}= \alpha$ so $\alpha = 1$. Hence the equation {\bf[T]} is actually written as
\[ C(y^{\beta},z)^{p} y^{p} x y^{p'}C(y^{\beta},z)^{p'} = C(y,z)^{p} x C(y,z)^{p'}. \]
This equality holds if and only if two equalities
\begin{equation}
\label{eqn:A}  \left\{
\begin{array}{l}
C(y^{\beta},z)^{p} y^{p} = C(y,z)^{p} \\
y^{p'}C(y^{\beta},z)^{p'} =C(y,z)^{p'} 
\end{array}
\right.
\end{equation}
hold. 
Let $m$ be the exponent sum of the $y$ in $C(y,z)$. By comparing the exponent of $y$ in (\ref{eqn:A}), we get
\begin{gather*}
\begin{cases} 
\beta mp+p=mp \\
\beta mp'+p'=mp'. 
\end{cases}
\end{gather*}
If $p\neq 0$, then $\beta m+1=m$. Since $\beta =\pm 1$ and $m$ is an integer, this is impossible so $p=0$. Similalry, $p'=0$ hence we conclude $A \equiv a$.

To determine $C$ and $D$ we consider the following two cases according to the value of $\beta$. Let $\kappa:F_{2}\rightarrow F_{2}$ be an automorphism defined by $\kappa(a)=C$ and $\kappa(b)=D$.\\

{\bf Case 1: $\beta=+1$.}\\

By {\bf[M]} and {\bf[B]}, 
\[ \left\{
\begin{array}{l}
C(y,z) = C(C(y,z),D(y,z))\\
D(y,z) = D(C(y,z),D(y,z))
\end{array}\right.
\] 
These equalities show that $\kappa^{2}=\kappa$, so $\kappa = id$.
Thus we get
\begin{itemize}
\item[$(T)$]  $(A,B,C,D) = (a,b,a,b)$.\\
\end{itemize}

{\bf Case 2: $\beta =-1$.}\\

By {\bf[M]} and {\bf[B]}, 
\[ \left\{
\begin{array}{l}
C(y^{-1},z)^{-1} = C(C(y,z)^{-1},D(y,z))\\
D(y^{-1},z) = D (C(y,z)^{-1},D(y,z)).
\end{array}\right.
\] 
Let $\iota:F_{2} \rightarrow F_{2}$ be the involution defined by $\iota(a)=a^{-1}$, $\iota(b)=b$. Then the above equalities say that $\iota\kappa\iota = \kappa \iota\kappa$. Moreover, $\kappa$ has order at most two since  $\kappa^{2}=\kappa^{2}\iota\kappa \kappa^{-1} \iota = \iota \kappa \kappa^{-1} \iota = id$. 

Let $G$ be the subgroup of $\Aut(F_{2})$ generated by $\iota$ and $\kappa$. 
The relations $\kappa^{2}=\iota^{2}=id$ and $\iota\kappa\iota = \kappa \iota\kappa$ show that $G$ is a quotient of the symmetric group $S_{3}$ so $G$ is either $\Z_{2}$ or $S_{3}$.

Let $\pi: \Aut(F_{2}) \rightarrow \textrm{GL}(2;\Z)$ be the projection.
By direct calculation, $\pi(\kappa)=\left(\begin{array}{cc} \pm 1  &0 \\ 0 & \pm 1  \end{array}\right)$, so $\pi(G) \cong \Z_{2}$.  
On the other hand, an order $3$ element of $\Aut(F_{2})$ is conjugate to $\rho:F_{2} \rightarrow F_{2}$ given by $\rho(a)=a^{-1}b, \rho(b)=a^{-1}$ \cite{mes}. This shows that $G$ cannot contain an element of order three, so $G$ is isomorphic to $\Z_{2}$. Hence $\kappa = \iota$ and we get 
\begin{itemize}
\item[$(T')$]  $(A,B,C,D) = (a,b^{-1},a^{-1},b)$.
\end{itemize}
\end{proof}

\subsection{The case $B$ contains $a^{\pm 1}$}

In this section we prove Proposition \ref{prop:caseIII}, which treats the case $B$ contains at least one $a^{\pm 1}$.
If $C$ does not contain $b^{\pm 1}$, then the swap-dual of $\Theta$ satisfies the assumption of Proposition \ref{prop:caseI}, so we may also assume that $C$ contains at least one $b^{\pm 1}$.

\begin{proposition}
 \label{prop:caseII}
Under the assumption of Proposition \ref{prop:caseIII}, $A$ contains at least one $b^{\pm 1}$ and $D$ contains at least one $a^{\pm 1}$
\end{proposition}

\begin{proof}

We show $A$ contains at least one $b^{\pm 1}$. The assertion for $D$ is obtained by considering the swap-dual. 

Assume contrary, $A$ does not contain $b^{\pm 1}$, so $A \equiv a^{\alpha}$ $(\alpha =\pm 1)$.
By the same argument as Proposition \ref{prop:caseI}, by \textbf{[T]} and \textbf{[Aut]} $\alpha= +1$ and $B \equiv a^{q}b^{\beta}a^{q'}$ $(\beta \in \{\pm 1\},\; q,q' \in \Z )$. Since $B$ contains at least one $a^{\pm 1}$, $qq' \neq 0$. By taking the backward-dual if necessary, we may assume that $q\neq 0$. Then the equation \textbf{[M]} is written as
\begin{equation}
\label{eqn:caseIIM}
x^{q}C(x^{q}y^{\beta}x^{q'},z)^{\beta}x^{q'} = C(x^{q}C(y,z)^{\beta}x^{q'},D(y,z))
\end{equation}

We show that the equation (\ref{eqn:caseIIM}) has no solutions.

Let $m$ be the exponent sum of $a$ in the word $C$.
By comparing the exponent of $x$ in (\ref{eqn:caseIIM}),
we have $(q+q')(m\beta +1) = m(q+q')$.
If $q+q' \neq 0$, then we get $m\beta+1 = m$. Since $m$ is an integer and $\beta \in \{ \pm 1\}$, this cannot happen so $q'=-q$.

Let us put
\[ C(x,y) \equiv x^{c}y^{r}*\cdots \equiv *\cdots y^{r'}x^{c'} \ \ (r,r' \neq 0). \]
Then the left side of (\ref{eqn:caseIIM}) is 
\begin{gather*}
x^{q}C(x^{q}y^{\beta}x^{-q},z)^{\beta}x^{-q} =
\begin{cases}
x^{2q}y^{c}x^{-q}z^{r}*\cdots = \cdots *z^{r'}x^{q}y^{c'}x^{-2q} & (\beta = +1)\\
x^{2q}y^{-c'}x^{-q}z^{r'}*\cdots = \cdots *z^{-r}x^{q}y^{-c}x^{-2q}& (\beta = -1)\\
\end{cases}
\end{gather*}
whereas the right side of (\ref{eqn:caseIIM}) is
\[ C(x^{q}C(y,z)^{\beta}x^{-q},D(y,z)) = 
x^{q}C(y,z)^{\beta c}x^{-q}D(y,z)^{r}\cdots = \cdots D(y,z)^{r'}x^{q}C(y,z)^{\beta c'}x^{-q}.
\]
By comparing the reduced word expression of both sides, we conclude (\ref{eqn:caseIIM}) has no solutions. For example, assume that $\beta = -1$, $c=0$ and $c' \neq 0$. Then the reduced word expression of the left side of (\ref{eqn:caseIIM}) is $x^{2q}*\cdots$, whereas the right side of (\ref{eqn:caseIIM}) begins with a letter $y^{\pm 1}$ or $z^{\pm 1}$, contradiction. The other cases are proven similarly.

\end{proof}

Summarizing, we have the following constraints for word $A,B,C$ and $D$:
\begin{gather}
\label{eqn:hypo}
\begin{cases}
A \text{ contains at least one } b^{\pm1}.\\
B \text{ contains at least one } a^{\pm1}.\\
C \text{ contains at least one } b^{\pm1}.\\
D \text{ contains at least one } a^{\pm1}.
\end{cases}
\end{gather}

To proceed further, we observe the following much stronger constraints for the words $B$ and $C$.
\begin{lemma}
\label{lem:restriction}
Under the assumption of Proposition \ref{prop:caseIII}, the reduced words $B$ and $C$ satisfy the following.
\begin{enumerate}
\item  $B \equiv b^{q}a^{\beta} b^{q'}, \;\;\; C \equiv a^{r}b^{\gamma} a^{r'}$  $(q,q',r,r' \in \Z,\  \beta,\gamma \in \{\pm 1\})$.
\item  $qq'rr'=0$.
\end{enumerate}
\end{lemma}
\begin{proof}
By (\ref{eqn:hypo}), we may put
\begin{gather*}
\begin{cases}
A \equiv a^{p}b^{p'}*\cdots & (p' \neq 0) \\
C \equiv a^{r}b^{\gamma}a^{r'}b^{\gamma'} *\cdots & (\gamma \neq 0).
\end{cases}
\end{gather*}

Assume that $r', \gamma' \neq 0$. Then $\red (C(x,y)^{p'}) \equiv x^{r_1} y^{\gamma_1} x^{r_2}y^{\gamma_2}*\cdots $ $(r_1,r_2,\gamma_1,\gamma_2 \neq 0)$, so
the right side of {\bf [T]} is
\[
\red (A(x,C(y,z)) \equiv x^{p}y^{r_1} z^{\gamma_1} y^{r_2}z^{\gamma_2}*\cdots.
\]
On the other hand, by Lemma \ref{lemma:key}, the left side of \textbf{[T]} is 
\[
 \red (A(A(x,y),C(B(x,y),z))) = \red( A(x,y)^{p}B(x,y)^{r_{1}}) z^{\gamma_1} \red(B(x,y)^{r_{2}}) x^{\gamma_2}* \cdots,
\]
so $B(x,y)^{r_{2}} = y^{r_{2}}$. Since $r_{2} \neq 0$, this implies $B(x,y) \equiv y$, which contradicts (\ref{eqn:hypo}). Therefore $\gamma'$ should be zero, so $C \equiv a^{r}b^{\gamma}a^{r'}$. By considering the swap-dual we conclude $B \equiv b^{q}a^{\beta}b^{q'}$.

To show (2), assume contrary that $qq'rr' \neq 0$. By Lemma \ref{lemma:key}, the right side of \textbf{[M]} is given by
\[
\red ( B(A(x,y),C(B(x,y),z) )
\equiv \left\{
\begin{array}{ll}
\red( B(x,y)^{r}) z^{\gamma} *\cdots & (q>0) \\
\red( B(x,y)^{-r'}) z^{-\gamma} *\cdots & (q<0).\\  
\end{array}
\right.
\]
For a non-zero integer $M$, $\red( B(x,y)^{M}) \equiv y^{Q}x^{\pm l}*\cdots$, where $Q= q$ or $-q'$, so the right side of \textbf{[M]} is written as
\[ 
\red ( B(A(x,y),C(B(x,y),z) ) \equiv y^{Q}x^{\pm 1}*\cdots \;\;\;\; (Q = q,-q')
\]

On the other hand, by Lemma \ref{lemma:key}, the left side of \textbf{[M]} is of the form
\[
\red C ( B(x, C(y,z) ), D(y,z) )
\equiv \left\{
\begin{array}{ll}
\red( C(y,z)^{q}) x^{\pm1} *\cdots & (r>0) \\
\red( C(y,z)^{-q'}) x^{\pm1} *\cdots & (r<0)\\  
\end{array}
\right.
\]
hence
\[ y^{Q} = C(y,z)^{Q'} \ \ \ (Q,Q' \in \{q,-q'\}). \]
Since $q,q'\neq 0$, this shows that $C(y,z)$ contains no $z^{\pm 1}$, which contradicts (\ref{eqn:hypo}).
\end{proof}

Now we are ready to complete the proof of Proposition \ref{prop:caseIII}.

\begin{proof}[Proof of Proposition \ref{prop:caseIII}]
By Lemma \ref{lem:restriction}, we already know $B \equiv b^{q}a^{\beta} b^{q'}$ and $C \equiv a^{r}b^{\gamma} a^{r'}$ $(q,q',r,r' \in \Z, \ \beta,\gamma \in \{\pm 1\})$. Since $qq'rr'=0$, by taking the swap and backward duals of $\Theta$ if necessary, we may assume that $q'=0$, namely, $B\equiv b^{q}a^{\beta}$.\\

{\bf Case 1: $q=0$.}\\

In this case, $B\equiv a^{\beta}$ $(\beta = \pm 1)$.
By \textbf{[B]},
\begin{equation}
\label{eqn:III-D}
 D(x^{\beta},z)= D (x^{\beta},D(y,z)).
\end{equation}

Assume that $D$ contains $b^{\pm 1}$.
Then the right side of (\ref{eqn:III-D}) contains at least one $y^{\pm 1}$, but the left side does not, contradiction. Thus $D$ cannot contain $b^{\pm 1}$, so by \textbf{[Aut]}, $D \equiv a^{\delta}$ $(\delta= \pm 1)$.

Now the equation \textbf{[M]} is written as
\begin{equation}
\label{eqn:III-A}
 A(x,y)^{\beta} = x^{\beta r}y^{\gamma\delta}x^{ \beta r'}.
\end{equation}

Now we study the following four cases according to the value $\beta$ and $\gamma\delta$.\\

{\it Subcase 1-1: $\beta=+1$ and $\gamma\delta=+1$.}\\

By (\ref{eqn:III-A}), $A(x,y) \equiv x^{r}yx^{r'}$.
The equation \textbf{[T]} is written as  
\[ (x^{r}yx^{r'})^{r}x^{r}z^{\gamma}x^{r'}  (x^{r}yx^{r'})^{r'} = x^{r}y^{r}z^{\gamma}y^{r'}x^{r'}. \]
This is satisfied if and only if the following two equalities hold.
\begin{gather*}
\begin{cases}
(x^{r}yx^{r'})^{r}x^{r} = x^{r}y^{r}\\
x^{r'} (x^{r}yx^{r'})^{r'} = y^{r'}x^{r'}
\end{cases}
\end{gather*}

These equations are written as
\begin{gather*}
\begin{cases}
(x^{r}yx^{r'})^{r} = (x^{r}yx^{-r})^{r}\\
(x^{r}yx^{r'})^{r'} = (x^{-r'}yx^{r'})^{r'},
\end{cases}
\end{gather*}
so \textbf{[T]} is satisfied if and only if $r'=-r$, and we conclude
\begin{enumerate}
\item[($A_{r}-1$)]  $\;\;\;(A,B,C,D) = (a^{r}ba^{-r}, a , a^{r}ba^{-r},a) \;\;\; (r \in \Z_{\geq 0}) $.
\item[($A_{r}-2$)]  $\;\;\;(A,B,C,D) = (a^{r}ba^{-r}, a , a^{r}b^{-1}a^{-r},a^{-1}) \;\;\; (r \in \Z_{\geq 0})$. 
\end{enumerate}
Here we remark that $(A_{-r}-1)=\overleftarrow{(A_{r}-1)}$ and $(A_{-r}-2)=\overleftarrow{(A_{r}-2)}$ so up to natural symmetry we can take $r \geq 0$.\\

{\it Subcase 1-2: $\beta=+1$ and $\gamma\delta=-1$.}\\

By (\ref{eqn:III-A}), $A(x,y) \equiv x^{r}y^{-1}x^{r'}$ so \textbf{[T]} is written as  
\[ (x^{r}y^{-1}x^{r'})^{r}x^{-r'}z^{-\gamma}x^{-r} (x^{r}y^{-1}x^{r'})^{r'} = x^{r}y^{-r'}z^{-\gamma}y^{-r}x^{r'}. \]
Consequently we get two equations
\begin{gather*}
\begin{cases}
(x^{r}y^{-1}x^{r'})^{r}x^{-r'} = x^{r}y^{-r'}\\
x^{-r} (x^{r}y^{-1}x^{r'})^{r'} = y^{-r}x^{r'}
\end{cases}
\end{gather*}

The first equation is written as
\begin{equation}
\label{eqn:III-1-2}
(x^{r}y^{-1}x^{r'})^{r} = x^{r}y^{-r'}x^{r'}. 
\end{equation}

If $r'=-r$, then (\ref{eqn:III-1-2}) is written as
\[ (x^{r}y^{-1}x^{-r})^{r} = (x^{r}y^{-1}x^{-r})^{-r}. \]
This shows $-r=r$, so $r=r'=0$.

If $r \neq -r'$, then the right side of (\ref{eqn:III-1-2}) is primitive, so $|r|\leq 1$. By direct calculations (\ref{eqn:III-1-2}) is satisfied only if $(r,r')=(1,1)$ or $(r,r')=(0,0)$. 
In the case $(r,r')=(0,0)$ we get
\begin{enumerate}
\item[($B-1$)] $\;\;\; (A,B,C,D) = (b^{-1}, a , b ^{-1}, a)$.
\item[($B-2$)] $\;\;\;(A,B,C,D) = (b^{-1}, a , b, a^{-1})$.
\end{enumerate}
and in the case $(r,r')=(1,1)$ we get
\begin{enumerate}
\item[($C-1$)] $\;\;\; (A,B,C,D) = (ab^{-1}a,a , ab^{-1}a, a)$.
\item[($C-2$)] $\;\;\; (A,B,C,D) = (ab^{-1}a,a , aba, a^{-1})$. \\
\end{enumerate}

{\it Subcase 1-3: $\beta = -1$ and $\delta\gamma=+1$.}\\

By (\ref{eqn:III-A}), $A(x,y) \equiv x^{r'}y^{-1}x^{r}$ so \textbf{[T]} is written as  
\[ (x^{r'}y^{-1}x^{r})^{r'}x^{r'}z^{-\gamma}x^{r} (x^{r'}y^{-1}x^{r})^{r} = x^{r'}y^{-r'}z^{-\gamma}y^{-r}x^{r}. \]
Consequently we get two equations
\begin{gather*}
\begin{cases}
(x^{r'}y^{-1}x^{r})^{r'}x^{r'} = x^{r'}y^{-r'}\\
x^{r} (x^{r'}y^{-1}x^{r})^{r} = y^{-r}x^{r}.
\end{cases}
\end{gather*}

As in the Subcase 1-1, these two equations are written as
\begin{gather*}
\begin{cases}
(x^{r'}y^{-1}x^{r})^{r'} = (x^{r'}y^{-1}x^{-r'})^{r'}\\
(x^{r'}y^{-1}x^{r})^{r} = (x^{-r}y^{-1}x^{r})^{r}.
\end{cases}
\end{gather*}
This shows $r=-r'$ and we get 
\begin{enumerate}
\item [$(A_{r}-2)^{-\sigma}$] $\;\;\; (A,B,C,D) = (a^{-r}b^{-1}a^{r},a^{-1},a^{r}ba^{-r},a).$ 
\item [$(A_{r}-3)$] $\;\;\; (A,B,C,D) = (a^{-r}b^{-1}a^{r},a^{-1},a^{r}b^{-1}a^{-r},a^{-1}). $ \\
\end{enumerate}

{\it Subcase 1-4: $\beta=-1$, $\gamma\delta=-1$.}\\

By (\ref{eqn:III-A}), $A(x,y) \equiv x^{r'}yx^{r}$ so the equation \textbf{[T]} is written as  
\[ (x^{r'}yx^{r})^{r'}x^{-r}z^{\gamma}x^{-r'} (x^{r'}yx^{r})^{r} = x^{r'}y^{r}z^{\gamma}y^{r'}x^{r}.\]
Consequently we get two equations
\begin{gather*}
\begin{cases}
(x^{r'}yx^{r})^{r'}x^{-r} = x^{r'}y^{r}\\
x^{-r'} (x^{r'}yx^{r})^{r} = y^{r'}x^{r}
\end{cases}
\end{gather*}

The first equation is written as
\begin{equation}
\label{eqn:III-1-4}
(x^{r'}yx^{r})^{r'} = x^{r'}y^{r}x^{r}
\end{equation}

If $r'=-r$, then (\ref{eqn:III-1-4}) is written as
\[ (x^{-r}yx^{r})^{-r} = (x^{-r}yx^{-r})^{r}. \]
Thus $-r=r$, so we conclude $r=r'=0$.
If $r \neq -r'$, then the right side of (\ref{eqn:III-1-4}) is primitive so $|r|\leq 1$. By direct calculation, (\ref{eqn:III-1-4}) is satisfied only if $(r,r')=(1,1)$ or $(r,r')=(0,0)$. 

If $(r,r')=(0,0)$ then we get 
\begin{itemize}
\item[$(B-2)^{-}$] $(A,B,C,D) = (b,a^{-1},b^{-1},a)$.
\item[$(B-1)^{\sigma}$] $(A,B,C,D) = (b,a^{-1},b,a^{-1})$.
\end{itemize}
and if $(r,r')=(1,1)$ and $(c,d)=(-1,1)$, then we get
\begin{enumerate}
\item[$(C-2)^{-\sigma}$] $(A,B,C,D) = (aba, a^{-1},ab^{-1}a, a)$.
\item[$(C-3)$] $ (A,B,C,D) = (aba, a^{-1},aba, a^{-1})$.\\
\end{enumerate}

{\bf Case 2: $q \neq 0$.}\\

Let us put $B \equiv y^{q}x^{\beta}$ and $D \equiv b^{s}a^{\delta}b^{s'}a^{\delta'}b^{s''}a^{\delta''}*\cdots$ $(\delta \neq 0)$. 
\begin{claim}
$\delta,\delta',\delta'',\ldots \in \{\pm 1,0\}$.
\end{claim}
\begin{proof}[Proof of Claim]
Assume that $|\delta|>1$. Then the left side of {\bf [B]} is 
\begin{equation}
\label{eqn:III2-Bl}
\red( D(B(x,y),z)) 
 \equiv \left\{ \begin{array}{ll} 
z^{s}y^{q}x^{\delta}y^{q}x^{\delta}*\cdots & (\delta>1) \\
z^{s}x^{-\delta}y^{-q}x^{-\delta}*\cdots & (\delta<-1). \\
\end{array}  
\right.
\end{equation}

On the other hand by Lemma \ref{lemma:key}, the right side of \textbf{[B]} is \begin{equation}
\label{eqn:III2-Br}
\red( D(B(x,C(y,z)),D(y,z)))
\equiv 
\left\{ \begin{array}{ll} 
\red(z^{s}C(y,z)^{q})x^{\delta}\red( C(y,z)^{q}) x^{\delta}*\cdots & (\delta>1) \\
z^{s} x^{-\delta} \red( C(y,z)^{-q})x^{-\delta}*\cdots & (\delta<-1). \\
\end{array}  
\right.
\end{equation} 

Hence by \textbf{[B]} $y^{q}=C(y,z)^{q}$.
 Since we have assumed $q\neq 0$, this implies $C(y,z) \equiv y$, which contradicts (\ref{eqn:hypo}). This proves $\delta= \pm 1$. 
Similar arguments show $\delta',\delta'',\ldots \in \{\pm 1,0\}$.
\end{proof}

Then we consider the following two cases according to the value of $\delta$.\\

{\it Subcase 2-1: $\delta=+1$}\\

\begin{claim}
\label{claim:formD}
If $\delta=1$, then $D \equiv b^{s}a$ $(s\neq 0)$.
\end{claim}
\begin{proof}[Proof of Claim]

If $\delta=+1$, then by (\ref{eqn:III2-Bl}) and (\ref{eqn:III2-Br}) we get 
\begin{equation}
\label{eqn:c-2-1} z^{s}y^{\delta} = D(y,z)^{s}C(y,z)^{\delta}. 
\end{equation}

If $s=0$, then the equation (\ref{eqn:c-2-1}) gives $y^{q}=C(y,z)^{q}$ $(q\neq 0)$, contradictions (\ref{eqn:hypo}), hence $s\neq 0$.

Assume that $\delta'=-1$. This implies, in particular, $s'\neq 0$. Then the right side of \textbf{[B]} is written as
\[ \red( D(B(x,y),z)) \equiv z^{s}y^{q}x^{\beta}z^{s'}x^{-\beta}*\cdots \]
and the left side of equation \textbf{[B]} is
\[ \red( D(B(x,C(y,z)),D(y,z)))
\equiv 
 \red (D(y,z)^{s}C(y,z)^{q}) x^{\beta}\red (D(y,z)^{s'}) x^{-\beta}* \cdots \]
so we get 
\[ D(y,z)^{s'} = z^{s'}. \]
Since $s'\neq 0$, this implies $D(y,z)= z$, which contradicts (\ref{eqn:hypo}).

Similarly, if $\delta'=0$, then the equation \textbf{[B]} is written as
\[ z^{s}y^{q}x^{\beta} z^{s'} =  D(y,z)^{s}C(y,z)^{q}x^{\beta}D(y,z)^{s'} \]
so we get 
\[ D(y,z)^{s'} = z^{s'}, \]
which leads to a contradiction unless $s'=0$, so we have $D \equiv b^{s}a$ as desired.

Finally assume that $\delta'=+1$. Then $s' \neq 0$, and equation \textbf{[B]} is written as
\[ z^{s}y^{q}x^{\beta} z^{s'}y^{q}x^{\beta} *\cdots =D(y,z)^{s}C(y,z)^{q}x^{\beta}D(y,z)^{s'}C(y,z)^{q}x^{\beta} *\cdots. \]
Hence we get 
\[ z^{s'}y^{q} = D(y,z)^{s'}C(y,z)^{q}.\]

By applying this equality to equation (\ref{eqn:c-2-1}), we get \[ D(y,z)^{s-s'}=z^{s-s'}.\]
If $s-s' \neq 0$, then $D(y,z) \equiv z$, which contradicts (\ref{eqn:hypo}) hence $s=s'$.

By iterating the above arguments for $\delta'',\delta''',\ldots$, we conclude that the possible form of the word $D(x,y)$ is $D(x,y)=y^{s}xy^{s}x \cdots y^{s}x=(y^{s}x)^{N}$.
By \textbf{[Aut]} $D(y,z)$ must be a primitive element, so $N=1$. This complete the proof of claim.
\end{proof}

By Claim \ref{claim:formD}, the equation (\ref{eqn:c-2-1}) is now in a simple form,
\begin{equation}
\label{eqn:c-2-1fin} 
z^{s}y^{q} = (z^{s}y)^{s}(y^{r}z^{\gamma}y^{r'})^{q}.
\end{equation}
This equation can be solved directly.

\begin{claim}
\label{claim:solution}
The equation (\ref{eqn:c-2-1fin}) is satisfied only if  $(q,r,\gamma,r',s)=(2,-1,-1,1,2)$.
\end{claim}
\begin{proof}[Proof of Claim]
We have seen that $C \equiv a^{r}b^{\gamma}a^{r'}$ and $D \equiv b^{s}a$.

First we consider the case $r=-r'$. Then the equation (\ref{eqn:c-2-1fin}) is written as  
\begin{equation}
\label{eqn:claim1} z^{s}y^{q} = (z^{s}y)^{s}y^{r}z^{\gamma q}y^{-r}
\end{equation}
If $r=0$, then $z^{s}y^{q} = (z^{s}y)^{s} z^{\gamma q}$. However, since $\gamma ,q \neq 0$ this is impossible so $r\neq 0$. Then we conclude $s>0$ and $r=-1$ because otherwise the right side of (\ref{eqn:claim1}) is already reduced so (\ref{eqn:claim1}) cannot be satisfied.
If $s=1$, then we get $z^{s}y^{q}=z^{s+ \gamma q}y^{-r}$. However, $\gamma,q \neq 0$ this is impossible.
If $s=2$, we get a solution $(q,r,\gamma,r',s)=(2,-1,-1,1,2)$.
Finally, if $s>2$ then $z^{s}y^{q}=(z^{s}y)^{s-2}z^{s}yz^{s+\gamma q}y$, it is impossible.

To complete the proof of claim, we show that there are no $(q,r,\gamma,r',s)$ satisfying (\ref{eqn:c-2-1fin}) if $r\neq -r'$. We prove this for the case $s>0,q>0$. The other cases are treated in a similar way.

In the case $s>0,q>0$ the equation (\ref{eqn:c-2-1fin}) is written as
\begin{equation}
\label{eqn:c-2-1claim} z^{s}y^{q} = (z^{s}y)^{s-1}z^{s}y^{r+1}z^{\gamma}y^{r'}(y^{r}z^{\gamma}y^{r'})^{q-1}.
\end{equation}
Since $r \neq -r'$, to satisfy (\ref{eqn:c-2-1claim}), we need $r =-1$ because otherwise the right side of (\ref{eqn:c-2-1claim}) is reduced. Hence we have
\[ z^{s}y^{q} = (z^{s}y)^{s-1}z^{s+\gamma}y^{r'}(y^{-1}z^{\gamma}y^{r'})^{q-1}\]

By \textbf{[Aut]}, $s$ and $\gamma$ should be coprime so $s + \gamma \neq 0$ if $|s| >1$. Hence if $s>1$, then the right side cannot be reduced further, so this equation cannot be satisfied.
Thus $s=1$, and we finally get
\[ zy^{q} = z^{1+\gamma}y^{r'}(y^{-1}z^{\gamma}y^{r'})^{q-1}. \]
If $q>1$, then the right side cannot be reduced further, and if $q=1$ we get $r'=+1$, which contradicts our assumption that $r \neq -r'$ because we have seen $r=-1$.
\end{proof}

By Claim \ref{claim:solution}, $B \equiv b^{2}a^{\beta}$, $C \equiv a^{-1}ba$ and $D\equiv b^{2}a$ ($\beta \in \{\pm 1 \}$). By \textbf{[M]}, $A^{\beta} = x^{\beta} y^{-1}x^{\beta}$, so we conclude
\begin{enumerate}
\item[$(D-1)$] $(A,B,C,D) = ( a^{-1}b^{-1}a, b^{2}a, a^{-1}b^{-1}a, b^{2}a)$.
\item[($D-2$)] $(A,B,C,D) = (aba^{-1},b^{2}a^{-1}, a^{-1}b^{-1}a, b^{2}a)$.\\ 
\end{enumerate}

{\it Subcase 2-2: $\delta=-1$.}\\

 In this case the equation \textbf{[B]} is given by
\[ z^{s}x^{-\beta}y^{-q} z^{s'}(y^{q}x^{\beta})^{\delta'} \cdots =D(y,z)^{s}x^{-\beta}C(y,z)^{-q}D(y,z)^{s'}(C(y,z)^{q}x^{\beta})^{\delta'} \cdots \]
 hence we get 
\[ z^{s} = D(y,z)^{s}\]
If $s \neq 0$, this shows that $D(y,z) \equiv z$, which contradicts (\ref{eqn:hypo}). Hence $s=0$ and $D(x,y)$ is 
\[ D(x,y) \equiv x^{-1}y^{s'}x^{\delta'}y^{s''} \cdots. \]
We consider the following three cases according to the value $\delta'$.\\

{\it Subcase 2-2-1: $\delta'=+1$.}\\

In this case $s' \neq 0$ and the equation \textbf{[B]} is 
\[ x^{-\beta}y^{-q} z^{s'} y^{q}x^{\beta} z^{s''} (y^{q}x^{\beta})^{\delta''} \cdots =x^{-\beta}C(y,z)^{-q}D(y,z)^{s'}C(y,z)^{q}x^{\beta}D(y,z)^{s''} (C(y,z)^{q}x^{\beta})^{\delta''}\cdots \]
so we have 
\begin{equation}
\label{eqn:c-2-2-1} y^{-q}z^{s'}y^{q} = C(y,z)^{-q}D(y,z)^{s'}C(y,z)^{q}. 
\end{equation}
Moreover, if $s'' \neq 0$, then we further get 
\begin{gather}
\label{eqn:c-2-2-11}
\begin{cases}
z^{s''} = D(y,z)^{s''} & (\delta'' = -1,0)\\
z^{s''}y^{q} = D(y,z)^{s''}C(y,z)^{q} & (\delta'' = +1)\\
\end{cases}
\end{gather}
However, (\ref{eqn:c-2-2-1}) and (\ref{eqn:c-2-2-11}) show $D(y,z)=z$, which contradicts (\ref{eqn:hypo}). Hence $s''=\delta''=\cdots = 0$ and $D\equiv a^{-1}b^{s'}$.
Then (\ref{eqn:c-2-2-1}) is written as 
\[
y^{-q}z^{s'}y^{q} =
(y^{r}z^{\gamma}y^{r'})^{-q} (y^{-1}z^{s'})^{s'} (y^{r}z^{\gamma}y^{r'})^{q}
\]
By comparing the exponent of $y$, we conclude $s'=0$. However, since we have assumed $s' \neq 0$, this is a contradiction. \\

{\it Subcase 2-2-2: $\delta'=-1$. }\\

In this case $s' \neq 0$ and the equation \textbf{[B]} is 
\[ x^{-\beta}y^{-q}z^{s'}x^{-\beta}y^{-q}z^{s''}(y^{q}x^{\beta})^{\delta''} \cdots =x^{-\beta}C(y,z)^{-q}D(y,z)^{s'}x^{-\beta}C(y,z)^{-q}D(y,z)^{s''} \cdots,\]
so we get
\begin{equation}
\label{eqn:c-2-2-2}
y^{-q}z^{s'} = C(y,z)^{-q}D(y,z)^{s'}. 
\end{equation}
Then, by using a similar argument as in SubCase 2-1 (Claim \ref{claim:formD}), we conclude that $\delta''=0$ and $s'=s''$. Thus, $D \equiv a^{-1}b^{s'}a^{-1}b^{s'}$. However, by {\bf [Aut]} $D$ is primitive, contradiction.\\
 
{\it Subcase 2-2-3: $\delta'=0$.}\\

By \textbf{[B]}, we get equation
\begin{equation}
\label{eqn:solution'} y^{-q}z^{s'} = (y^{r}z^{\gamma}y^{r'})^{-q}(y^{-1}z^{s'})^{s'}. 
\end{equation}

By a similar argument as Claim \ref{claim:solution}, (\ref{eqn:solution'}) is satisfied only if $(q,r,\gamma,r',s') = (2,-1,1,1,2)$ so $C\equiv a^{-1}ba$ and $D\equiv a^{-1}b^{2}$. By \textbf{[M]}, $A(x,y)^{\beta}=x^{-\beta}y^{-1}x^{\beta}$, hence we conclude
\begin{enumerate}
\item[$(D-3)$] $(A,B,C,D) = (a^{-1}ba, b^{2}a, a^{-1}ba, a^{-1}b^{2})$.
\item[$(D-4)$] $(A,B,C,D) = (ab^{-1}a^{-1}, b^{2}a^{-1}, a^{-1}ba, a^{-1}b^{2})$.
\end{enumerate}
\end{proof}

\section{Group-valued invariants of knots and links}
\label{sec:invariants}

We close the paper by giving a short discussion on group-valued invariants of knots and links from local $\Aut(F_{n})$ representations.

\begin{definition}
We say a local $\Aut(F_{n})$ braid representation $\Theta$ defined by $(\tau_{1},\ldots, \tau_{n-1}) \in \Aut(F_{2})^{n-1}$ has the {\em stabilization properties} if the following two conditions are satisfied:
\begin{description}
\item[S1] The group $G_{i} = \langle a,b \: | \: (b)\tau_{i} = b \rangle$ is an infinite cyclic group generated by $a=b$.
\item[S2] For any $m>n$, there exist a local $\Aut(F_{m})$ representation $\Theta^{*}$ which extends $\Theta$: That is, there is a local $\Aut(F_{m})$ representation $\Theta^{*}$ defined by $(\tau_{1},\ldots, \tau_{n-1}, \tau_{n},\ldots, \tau_{m-1})$.
\end{description} 
\end{definition}

Let $L$ be an oriented link in $S^{3}$ represented as a closure of an $n$-braid $\beta$, and let $\Theta$ be a local $\Aut(F_{n})$ representation having the stabilization properties. Let $G_{\Theta}(\beta)$ be a group given by the presentation
\[ G_{\Theta}(\beta) = \left\langle x_{1},\ldots,x_{n} \: | \: (x_{i})[\Theta(\beta)] = x_{i} \;\;\; (i=1,\ldots, n) \right\rangle. \]

\begin{proposition}[Group-valued invariant of links]
\label{prop:ginvariant}
The group $G_{\Theta}(\beta)$ does not depend on a choice of closed braid representatives of $L$, hence $G_{\Theta}(\beta)$ defines a group-valued invariant of $L$.
\end{proposition}
\begin{proof}
By Markov theorem (see \cite{b}, for example) the closures of two braids $\alpha$ and $\beta$ represent the same oriented link if and only if $\alpha$ is converted to $\beta$ by applying following two operations.
\begin{description}
\item[Conjugation] $\alpha \rightarrow \gamma^{-1}\alpha\gamma$ where $\alpha,\gamma \in B_{n}$.
\item[(De)Stabilization] $\alpha \leftrightarrow \alpha\sigma_{n}^{\pm 1}$ where $\alpha \in B_{n}$.
\end{description}

Assume that $\Theta$ be a local $\Aut(F_{n})$ representation defined by $(\tau_{1},\ldots,\tau_{n-1}) \in \Aut(F_{2})^{n-1}$.
By definition, the group $G_{\Theta}(\beta)$ does not change under the conjugacy, so it is sufficient to show the invariance under the stabilization.

Take an arbitrary local $\Aut(F_{n+1})$ representation $\Theta^{*}$ which extends $\Theta$. That is, we assume that $\Theta^{*}$ is defined by $(\tau_{1},\ldots,\tau_{n-1},\tau_{n}) \in \Aut(F_{2})^{n}$ for some $\tau_{n} \in \Aut(F_{2})$. Such $\Theta^{*}$ exists from the assumption {\bf S2}. Then 
\[ [\Theta^{*}(\beta \sigma_{n}^{\pm 1})] (x_{i}) =
\left\{
\begin{array}{ll}
(x_{i})[\Theta(\beta)] & i \neq n,n+1, \\
(x_{n})[\Theta(\beta) \tau_{n}^{\pm 1}] & i=n, \\
(x_{n+1})\tau_{n}^{\pm 1} & i=n+1.
\end{array}
\right.
\]
By {\bf S1}, the relation $(x_{n+1})\tau_{n}^{\pm 1} = x_{n+1}$ implies $(x_{n})\tau_{n}^{\pm 1} = x_{n} = x_{n+1}$, hence two groups $G_{\Theta}(\beta)$ and $G_{\Theta^{*}}(\beta\sigma_{n}^{\pm 1})$ are isomorphic.
\end{proof}

A priori, it is not clear whether there is a local $\Aut(F_{n})$-representation having stabilization properties. However, the classification theorem shows almost all local $\Aut(F_{n})$-representations satisfy the stabilization properties.
\begin{corollary}
If $\Theta$ is a local $\Aut(F_{n})$-representation which is not of type $(T)$ or $(T')$, then $\Theta$ satisfies the stabilization properties.
\end{corollary}

To take into account of a choice of local $\Aut(F_{n})$ representations and closed braid representative, it is convenient to formulate the group-valued invariants as follows.

\begin{definition}[Group-valued invariants]
Let $L$ be an oriented link in $S^{3}$ and $\wTheta$ be a non-trivial local $\Aut(F_{\infty})$-representation defined by a sequence $(\tau_{1},\tau_{2},\ldots,)$. The \emph{group-valued invariant} $G_{\wTheta}(L)$ is a group $G_{\wTheta}(L) = G_{\Theta_{n}}(\beta)$
where $\beta$ is an $n$-braid whose closure is $L$ and $\Theta_{n}$ is a local $\Aut(F_{n})$-representation defined by $(\tau_{1},\ldots,\tau_{n})$. 
\end{definition}

As we have mentioned in introduction, at first glance we have obtained many new group-valued invariants. However, classification theorem also tells us that group-valued invariants for certain local $\Aut(F_{n})$-representations are nothing new.

\begin{theorem}
\label{theorem:ginvariant}
Two local $\Aut(F_{\infty})$ representations defines the same invariant if they are of the same type. In particular, a group-valued invariant from local $\Aut(F_{\infty})$ representations coincides with the group-valued invariant from the corresponding Wada-type representation. 
\end{theorem}
\begin{proof}
Assume that $\wTheta$ and $\wTheta'$ are of the same type defined by $(\tau_{1},\ldots)$ and $(\tau'_1,\ldots)$, respectively.
Represent $L$ a closure of $n$-braid $\beta$, and let $\Theta_{n}, \Theta'_{n}$ be the  local $\Aut(F_{n})$-representations defined by $(\tau_{1},\ldots,\tau_{n-1})$ and  $(\tau'_{1},\ldots,\tau'_{n-1})$, respectively.
From the Classification theorem of local $\Aut(F_{n})$-representation, by taking sufficiently large $N$, there is a local $\Aut(F_{N})$-representations $\Phi$ defined by a sequence
\[ (\underbrace{\tau_1,\ldots,\tau_{n-1}}_{n-1},\ldots,\underbrace{\tau'_{1},\ldots,\tau'_{n-1}}_{n-1}) \in \Aut(F_{2})^{N}. \]
Let $\textsf{sh}:B_{n}\rightarrow B_{N}$ be the homomorphism defined by $\textsf{sh}(\sigma_{i}) = \sigma_{i+(N-n)}$.
The closures of $N$-braids $(\beta \sigma_{n}\sigma_{n+1}\cdots\sigma_{N-1})$ and $\sigma_{1}\cdots \sigma_{N-n}\textsf{sh}(\beta)$ are $L$, hence 
\[
G_{\wTheta}(L) \cong G_{\Phi}(\beta \sigma_{n}\sigma_{n+1}\cdots\sigma_{N-1}) \cong G_{\Phi}(\sigma_{1}\cdots \sigma_{N-n}\textsf{sh}(\beta)) \cong G_{\wTheta'}(L).
\]
\end{proof}


\begin{thebibliography}{1}

\bibitem{b} J. Birman,
{\em{Braids, Links, and Mapping Class Groups,}}
Annals of Math. Studies \textbf{82}, Princeton Univ. Press (1974).
\bibitem{cjkls}
S. Carter, D. Jelsovsky, S. Kamada, L. Langford and M. Saito,
{\em{Quandle cohomology and state-sum invariants of knotted curves and surfaces,}} Trans. Amer. Math. Soc. \textbf{355}, (2003), 3947--3989. 
\bibitem{cp} J. Crisp and L. Paris,
{\em Representation of the braid group by automorphisms of groups, invariant of links, and Garside groups},
Pacific J. Math. \textbf{221}(2005) 1--27.
\bibitem{ddrw} P. Dehornoy, I.Dynnikov, D.Rolfsen and B.Wiest, 
{\em{Ordering Braids,}}
Mathematical Surveys and Monographs \textbf{148}, Amer. Math. Soc. 2008.
\bibitem{i1} T. Ito,
{\em{A functor-valued extension of knot quandles,}}
J. Math. Soc. Japan \textbf{64} (2012) 1147--1168.

\bibitem{i} T. Ito,
{\em{The classification of Wada-type representations of braid groups,}}
J. Pure Appl. Algebra, \textbf{217}, (2013), 1754--1763.

\bibitem{mes} S. Meskin, 
{\em{Periodic automorphisms of the two-generator free group,}}
Lecture Notes in Math., \textbf{372}, (1974), 494--498.

\bibitem{sh}
V. Shpilrain,
{\em{Representing braids by automorphisms,}} 
Internat. J. Algebra Comput. \textbf{11} (2001), 773--777. 

\bibitem{w} M. Wada,
{\em{Group invariants of links,}}
Topology, \textbf{31}, (1992), 399-406.
\end{thebibliography}
\end{document}